\documentclass
{amsart}

\usepackage{amsmath,amssymb,amsthm} 
\usepackage[all]{xy}
\usepackage{extarrows}
\usepackage{graphicx}
\usepackage{bm}
\usepackage{comment, url}
\usepackage[normalem]{ulem}

\usepackage{time}

\newtheorem{thm}{Theorem}[section]
\newtheorem{lem}[thm]{Lemma}

\newtheorem{prop}[thm]{Proposition}  

\theoremstyle{remark}

\theoremstyle{definition}

\newtheorem{rem}[thm]{Remark} 

\newtheorem{eg}[thm]{Examples}

\newtheorem{def/prop}[thm]{Definition/Proposition}

\numberwithin{equation}{section}

\usepackage[usenames]{color}

\newcommand{\red}{}

\def\det{\mathop{\mathrm{det}}\nolimits}

\def\Ker{\mathop{\mathrm{Ker}}\nolimits}

\def\Spec{\mathop{\mathrm{Spec}}\nolimits}

\def\mod{\mathop{\mathrm{mod}}\nolimits}

\newcommand{\bb}[1]{{\mathbb{#1}}}
\newcommand{\mca}[1]{{\mathcal{#1}}}

\newcommand{\To}{\longrightarrow}

\newcommand{\inj}{\hookrightarrow}
\newcommand{\surj}{\twoheadrightarrow}

\newcommand{\congto}{\overset{\cong}{\to}}

\newcommand{\N}{\bb{N}}
\newcommand{\Z}{\bb{Z}}
\newcommand{\Zp}{\bb{Z}_{p}}
\newcommand{\Q}{\bb{Q}}
\newcommand{\Qp}{\bb{Q}_{p}}
\newcommand{\R}{\bb{R}}

\newcommand{\Cp}{\bb{C}_{p}}

\newcommand{\ol}{\overline}

\newcommand{\ds}{\displaystyle}

\newcommand{\wh}[1]{{\widehat{#1}}}

\DeclareMathOperator*{\restprod}%
 {\mathchoice{\ooalign{\ensuremath{\displaystyle\prod}\crcr\ensuremath{\displaystyle\coprod}}}%
             {\ooalign{\ensuremath{\textstyle\prod}\crcr\ensuremath{\textstyle\coprod}}}%
             {\ooalign{\ensuremath{\scriptstyle\prod}\crcr\ensuremath{\scriptstyle\coprod}}}%
             {\ooalign{\ensuremath{\scriptscriptstyle\prod}\crcr\ensuremath{\scriptscriptstyle\coprod}}}%
 }

\newcommand{\pmx}[1]{\begin{pmatrix}#1\end{pmatrix}}
\newcommand{\spmx}[1]{{\small \pmx{#1}}}

\title[Profinite completions determine Alexander polynomials]
{The profinite completions of knot groups determine the Alexander polynomials}

\author{Jun Ueki}
\date{\today}


\begin{document}
\maketitle

\begin{abstract} 
We study several properties of the completed group ring $\wh{\Z}[[t^{\wh{\Z}}]]$ and the completed Alexander modules of knots. 
Then we prove that if the profinite completions of the groups of two knots $J$ and $K$ are isomorphic, then their Alexander polynomials $\Delta_J(t)$ and $\Delta_K(t)$ coincide. 
\end{abstract}

\tableofcontents 

\section{Introduction} 
It is experimentally known in several occasions that in order to distinguish two knots it is efficient to compare homology torsions of their finite covers  (e.g., \cite{Perko1974proc}, \cite{KodamaSakuma1992}). 
Since homology torsions of finite covers are described by the profinite completions of knot groups (see Remark \ref{homologytorsion}), it is an interesting question to ask what topological properties of knots are determined by the profinite completions of knot groups; 
in other words, what the inverse systems of finite quotients of knot groups know. 
In this article, we prove that the profinite completions of knot groups completely determine the Alexander polynomials of knots, in the sense of Theorem 1.1.

3-manifold groups $\pi$ are \emph{residually finite}, namely, canonically inject into their profinite completions $\wh{\pi}$ by results of Hempel and Perelman (\cite{Hempel1987}, \cite{PerelmanGC1}, \cite{PerelmanGC2}, \cite{PerelmanGC3}). 
Grothendieck wrote that it is an interesting question whether finitely generated (finitely presented) residually finite groups are determined by their profinite completions (\cite{Grothendieck1970profinite}), while negative examples of finitely presented groups were given by Bridson--Grunewald (\cite{BridsonGrunewald2004}). 
Earlier negative examples of not necessarily finitely presented groups had been given by  Platonov--Tavgen (\cite{PlatonovTavgen1986}). 

What topological properties are determined by $\wh{\pi}$'s is a very subtle problem and yet to be understood completely. 
More detailed background and related topics will be described in Section \ref{ssPre}.  

Let us recall prior results to our main result. 
Bridson and Reid distinguished $\wh{\pi}$ of the figure-eight knot from those of other 3-manifolds (\cite{BridsonReid2015}). 
Hence the Alexander polynomial of the figure-eight knot is determined by its $\wh{\pi}$. 
Boileau and Friedl proved among other statements that the Alexander polynomial of a knot is determined by $\wh{\pi}$ if it does not vanish at any root of unity  (\cite[Proposition 4.10]{BoileauFriedl2015}). 
They used Fox's formula for $\Z/n\Z$-covers (\cite{FoxFDC3}) and applied Fried's proposition (\cite{Fried1988}). We generalize their results by removing any assumption on knots. The following theorem makes precise what we mean by the statement in the title of this article that the profinite completions of knot groups determine the Alexander polynomials: 

\begin{thm} \label{theorem}
Let $J$ and $K$ be knots in $S^3$ and suppose that an isomorphism $\varphi: \wh{\pi}_1(S^3-J) \congto \wh{\pi}_1(S^3-K)$ between the profinite completions of their knot groups is given. 
Then their Alexander polynomials $\Delta_J(t)$ and $\Delta_K(t)$ coincide up to multiplication by a unit of $\Z[t^{\Z}]$. 
\end{thm} 

The idea of our proof is to improve the argument of \cite{BoileauFriedl2015}. 
We consider not just the orders of groups on each layer, but also an isomorphism between the completed Alexander modules over the completed group ring $\wh{\Z}[[t^\wh{\Z}]]$, and obtain the equality of the Fitting ideals. 
In Section \ref{ssAlg}, we study several properties of $\wh{\Z}[[t^\wh{\Z}]]$, which would be useful also in studies of $\wh{\Z}$-covers of links or $\wh{\Z}$-extensions of number fields (e.g., \cite{Ueki4}, \cite{Asada2008}). 
We especially prove that \emph{in the completed group ring $\wh{\Z}[[t^\wh{\Z}]]$, any element $0\neq f(t)\in \Z[t]$ is not a zero-divisor} (Lemma \ref{not-ann}).  
In Section \ref{ssTop}, we consider inverse systems of branched $\Z/n\Z$-covers of knots and obtain an equality of ideals in $\wh{\Z}[[t^\wh{\Z}]]$. In addition, we define and study the completed Alexander modules of knots. 
In Section \ref{ssPrf}, we prove our theorem. 

In this article, we denote the profinite integer ring $\varprojlim_n \Z/n\Z$ by $\wh{\Z}$ and the $p$-adic integer ring $\varprojlim_n \Z/p^n\Z$ by $\Zp$ for each prime number $p$. 

\section{Preliminaries} \label{ssPre} 

In order to put our work in the context, we survey some general background and related works, together with some future sight. 
We will not make use of them in the paper, other than the definition of profinite completion.
A basic literature of profinite groups is \cite{RibesZalesskii2010}.\\ 

\emph{The profinite completion} $\wh{\pi}$ of a discrete group $\pi$ is a topological group defined by $\varprojlim_\Gamma \pi/\Gamma$, where $\Gamma$ runs through all the normal subgroups of finite index, and endowed with the weakest topology such that the kernel $\ker (\wh{\pi}\surj \pi/\Gamma)$ of the natural projection is open for every $\Gamma$. 

A group $\pi$ is said to be \emph{residually finite} if each nontrivial $g\in \pi$ has a finite quotient of $\pi$ in which the image of $g$ is nontrivial. This condition is equivalent to that the canonical homomorphism $\pi\to \wh{\pi}$ is an injection. 

A residually finite group $\pi$ is said to be \emph{Grothendieck rigid} if none of its finitely generated proper subgroups $\Gamma<\pi$ induces an isomorphism $\wh{\Gamma}\congto \wh{\pi}$ on their profinite completions (\cite{LongReid2011G}). 
Grothendieck especially wrote that it is an interesting question whether 
every finitely presented residually finite group would satisfy this condition (\cite{Grothendieck1970profinite}), 
while negative examples were given by Bridson--Grunewald (\cite{BridsonGrunewald2004}). Thus ``$\wh{\pi}$ forgets about $\pi$ to some extent.''\\ 

By a result of Hempel (\cite{Hempel1987}) together with Perelman's solution to the geometrization conjecture (\cite{PerelmanGC1}, \cite{PerelmanGC2}, \cite{PerelmanGC3}), 
the fundamental group of any compact 3-manifold is residually finite. 
By Long--Reid (\cite{LongReid2011G}), the fundamental group of any closed geometric 3-manifold is Grothendieck rigid. 
In addition, recently Boileau and Friedl proved that the fundamental groups of compact, orientable, irreducible 3-manifolds with toroidal boundaries are Grothendieck rigid (\cite{BoileauFriedl2017}). 
However, it seems still unknown whether profinite completions of groups of distinct two knots are never isomorphic to each other. 

Now we focus on the question of what topological properties the profinite completions $\wh{\pi}$ of 3-manifolds groups know. 
Note that in this article, if we write that \emph{$\wh{\pi}$ determines the property $P$}, then it means the following statement: 
\emph{Suppose that $M$ and $N$ are 3-manifolds with $\wh{\pi}_1(M)\cong \wh{\pi}_1(N)$. Then $M$ satisfies the property $P$ if and only if so \red{does} $N$.}  
(In another context, it might mean instead that \emph{we can explicitly describe whether $M$ satisfies the property $P$ or not by using $\wh{\pi}_1(M)$}.)  

By Wilton--Zalesskii (\cite{WiltonZalesskii2014}), $\wh{\pi}$ of a closed 3-manifold $M$ determines whether $M$ is hyperbolic, and whether it is Seifert fibered. 
By Funar (\cite{Funar2013}) and Hempel (\cite{Hempel2014}), there are pairs of torus bundles and those of Seifert 3-manifolds whose fundamental groups are not isomorphic but whose $\wh{\pi}$'s are isomorphic, 
while the existence of such a pair of hyperbolic 3-manifolds is still unknown. 
Other recent progresses are due to Wilton--Zalesskii and Wilkes (\cite{WiltonZalesskii2010}, \cite{WiltonZalesskii-arXiv1703}, \cite{Wilkes2017Dedicata}, \cite{Wilkes2018JA}, \cite{Wilkes-arXiv1710}, \cite{Wilkes-arXiv1801}, \cite{Wilkes-arXiv1802}). 

In regard to knot exteriors, Bridson--Reid distinguished $\wh{\pi}$ of the figure-eight knot from those of other 3-manifolds (\cite{BridsonReid2015}). 
Boileau--Friedl distinguished $\wh{\pi}$ of each torus knot and the figure-eight knot from those of other knots (\cite{BoileauFriedl2015}). 
In addition, Bridson--Reid--Wilton proved for compact 3-manifolds with 1st Betti number 1 that fiberedness is determined by $\wh{\pi}$'s (\cite{BridsonReidWilton2017}), \red{and Jaikin-Zapirain removed the condition on Betti number (\cite{JaikinZapirain2017fibering}).}

As for an explicit description by using $\wh{\pi}$ of the Alexander polynomial of a knot, we have Hillar's study \cite{Hillar2005} so that we have an algorithm to recover a polynomial without root on roots of unity from its cyclic resultants after knowing its degree. 
After our study in this article, it will still remain for instance to study how to recover the Alexander polynomial of a knot $K$ without any assumption from the family of groups $\{\wh{H}_1(X_n)\}_n$ associated to the cyclic covers $\{X_n\to X\}_n$ over the knot exterior $X=S^3-K$. \\ 

An important application of profinite (pro-sol\red{vable}) completions of fundamental groups is the work of Friedl--Vidussi (\cite{FriedlVidussi2011}). They used pro-solvable completions to prove that twisted Alexander polynomials determine the fiberedness of 3-manifolds. 
Another application can be found in a study of $PD(3)$-groups by Boileau--Hillman (\cite{BoileauHillman-arXiv1710}).\\ 

Finally we would like to give a remark on \red{the} analogy between knots and prime numbers. 
It was initially pointed out by Mazur (\cite{Mazur1963}) that there is a close relation between Iwasawa theory on $\Zp$-extensions of number fields and Alexander--Fox theory on systems of cyclic covers over knot exteriors. 
After years, Kapranov, Reznikov, and Morishita described the analogy between low dimensional topology and number theory in a systematic manner, and their study is called \emph{arithmetic topology} (cf. \cite{Morishita2012}). 

One of the basic analogies is obserbed between the fundamental group $\pi_1(M)$ of a 3-manifold $M$ and the \'etale fundamental group $\pi_1^{\text{\'et}}(\Spec \mathcal{O}_k)$ of the integer ring $\mathcal{O}_k$ of a number field $k$, where the latter is a profinite group \emph{a priori}. 
\red{Therefore, the study of profinite rigidity of 3-manifold groups would give a new angle in arithmetic topology, as mentioned by Mazur in \cite[page 6]{Mazur2012}.}  

\red{We can expect further progresses in this direction. As for Alexander--Fox theory, 
twisted Alexander invariants of knots associated to certain profinite representations 
are investigated from a viewpoint of Hida--Mazur theory and Galois deformation theory (\cite{MTTU}, \cite{KMTT2017}), besides an analogue of Fox' formula for twisted Alexander polynomial is given by Tange (\cite{TangeRyoto1}). 
In addition, we have a remarkable theorem by Le (the Bergeron--Venkatesh conjecture) on the asymptotic formula of homology torsion growth in which hyperbolic volume appears (e.g., \cite{BV2013}, \cite{Le-vol}), while L\"uck's ``optimistic conjecture'' on $L^2$-torsion would imply that hyperbolic volume is determined by $\wh{\pi}$ (\cite{Luck2015slide}). }

\red{Moreover,} in \emph{anabelian geometry}, Mochizuki introduced the terms ``\emph{mono/bi-anabelian}'' in order to distinguish formulations in reconstruction problems for arithmetic fundamental groups (cf. \cite[Remarks 3.7.3, 3.7.5]{MochizukiTAAG3}). 
It 
\red{would be} interesting to examine how to formulate answers to our question of \emph{what $\wh{\pi}$ knows} in comparison with his point of view.

\section{Algebraic lemmas} \label{ssAlg}
To begin with, we recall two assertions which will be used in this section. 
A polynomial $f(t)=\sum_{0\leq i \leq d} a_it^{d-i}$ in $\Z[t]$ with $d=\deg f(t)$ 
is said to be \emph{reciprocal} if $a_i=a_{d-i}$ holds for every $i$. Such a polynomial is also said to be \emph{self-reciprocal} or \emph{palindromic}. 
For two polynomials $f(t)=\sum_{0\leq i \leq d} a_it^{d-i}$ and $g(t)=\sum_{0\leq j \leq e} b_jt^{e-j}$ in $\Z[t]$ with $d=\deg f(t)$ and $e=\deg g(t)$, their resultant $R(f(t),g(t))\in \Z$ is defined as the determinant of the Sylvester matrix 
$${\rm Syl}(f(x),g(x))=\spmx{a_0&a_1&\cdots&a_m&\\
&\ddots&\ddots&&\ddots&\\
&&a_0&a_1&\cdots&a_m\\
b_0&b_1&\cdots&b_n&\\
&\ddots&\ddots&&\ddots&\\
&&b_0&b_1&\cdots&b_n} \in {\rm M}_{d+e}(\Z),$$ 
whose entries are given by their coefficients $a_i$ and $b_i$. 
We have $R(f(t),g(t))=a_0^eb_0^d\prod_{i,j} (\alpha_i- \beta_j)$, where $\alpha_i$ and $\beta_j$ runs through roots of $f(t)$ and $g(t)$ in an algebraic closure $\ol{\Q}$ of $\Q$ (cf.~\cite{Weber1979}). 

Fried's proposition is stated as follows:  

\begin{prop}[Fried, {\cite[Proposition]{Fried1988}}] \label{Fried} 
Let $f(t)$ be a reciprocal polynomial in $\Z[t]$ and let $R(f(t),t^n-1)$ denote the resultant of $f(t)$ and $t^n-1$ for each $n\in \N$. If $R(f(t),t^n-1)\neq 0$ holds for every $n\in \N_{>0}$, then the sequence $\{|R(f(t),t^n-1)|\}_n$ determines $f(t)$.
\end{prop}

\red{We remark that Proposition \ref{Fried} was originally stated for $f(t) \in \R[t]$ and in Fried's proof the condition that the coefficients are in $\R$ was essential.}  
The proposition was proved by studying the zeta function $\ds B(z)=\red{\sum_{n=1}^{\infty}}\, |R(f(t),t^n-1)|\, \frac{z^n}{n}$ of dynamical system, which was introduced by Artin and Mazur (\cite{ArtinMazur1965Annals}). We cannot remove the assumption $R(f(t),t^n-1)\neq 0$ (e.g. Fried's pair in Example \ref{example}), while Hillar (\cite{Hillar2005}) proved without this assumption that $B(z)$ is a rational function. 
A recent development on \red{such kinds of} zeta functions is due to Br\"aunling (\cite{Braunling2017}). 

The following proposition is an algebraic generalization of Fox's formula: 

\begin{prop}[Weber, {\cite{Weber1979}}] \label{Weber} 
Let $f(t)$ and $g(t)$ be non-zero polynomials in $\Z[t]$ and suppose that the highest coefficient and the constant term of $g(t)$ are equal to $\pm1$. 
If $f(t)$ and $g(t)$ has no common root in an algebraic closure $\ol{\Q}$ of $\Q$, then 
$\Z[t]/(f(t),g(t))$ is a finite group with order $|R(f(t),g(t))|$. 
\end{prop}
 
Note that $R(f(t),t^n-1)=0$ holds if and only if $f(t)$ has a root at a primitive $m$-th root of unity for some $m\in \N$ with $m|n$. 
\emph{The $m$-th cyclotomic polynomial} $\Phi_m(t)\in \Z[t]$ is defined as the minimal polynomial of a primitive $m$-th root of unity over $\Q$. It vanishes at every primitive $m$-th root of unity and satisfies $\prod_{m|n}\Phi_m(t)=t^n-1$. 
Let $\ol{\Q}$ be an algebraic closure of $\Q$ and let $\zeta_n \in \ol{\Q}$ be a primitive $n$-th root of unity for each $n\in \N_{>0}$.

The completed group ring $\wh{\Z}[[t^\wh{\Z}]]$ is defined by $\varprojlim_n \wh{\Z}[t^{\Z/n\Z}]$, which can be identified with $\varprojlim_n \wh{\Z}[t]/(t^n-1)$. 
Since the composite $\wh{\Z}[t^\Z]\inj \wh{\Z}[[t^\wh{\Z}]]\surj \wh{\Z}[t^{\Z/n\Z}]$ is the natural surjection, 
the natural projection $\wh{\Z}[[t^\wh{\Z}]]\surj \wh{\Z}[t^{\Z/n\Z}]$ for each $n\in \N_{>0}$ is a surjection. 
We regard $\Z[t]$ as a subring of $\wh{\Z}[[t^\wh{\Z}]]$. 
We have natural decompositions $\wh{\Z}\cong \prod_p \Zp$ and $\wh{\Z}[[t^\wh{\Z}]]\cong \prod_p \Zp[[t^{\wh{\Z}}]]$. 
Indeed, let $m\in \N$ with the prime decomposition $m=\prod_i p_i^{e_i}$. Then we have $\Z/m\Z \cong \prod_i \Z/p_i^{e_i}\Z$ by the Chinese remainder theorem, and hence the decomposition of coefficients $\Z/m\Z[t^{\Z/n\Z}]\cong \prod_i (\Z/p_i^{e_i}\Z[t^{\Z/n\Z}])$ for each $n$. Since the inverse limit is compatible with product of sets on each layer, we have the desired isomorphisms.  
They are useful because $\Zp$ is an integral domain while $\wh{\Z}$ is not. For each prime number $p$, let $\Cp$ denote the completion of an algebraic closure of the $p$-adic numbers $\Qp$, and fix an embedding $\ol{\Q}\inj \Cp$.

Now we explain that we can substitute roots of unity for elements of $\wh{\Z}[[t^{\wh{\Z}}]]$. 
We have the natural surjection $\mod \Phi_m: \Zp[[t^\wh{\Z}]]\surj \Zp[t]/(\Phi_m(t))$ for each $m$. 
Indeed, for any $n$ with $m|n$, we have a natural map $\Zp[[t^{\wh{\Z}}]]\surj \Zp[t^{\Z/n\Z}]\cong \Zp[t]/(t^n-1)\surj \Zp[t]/(\Phi_m(t))$. Since $\{\Zp[t^{\Z/n\Z}]\}_n$ forms an inverse system, this map is independent of $n$. 

In each $\Zp[t]$, $\Phi_m(t)$ is not necessarily irreducible. 
For instance, if $m|(p-1)$, then $\Zp$ contains primitive $m$-th roots of unity, mainly due to Hensel's lemma (cf. \cite[p.112]{Gouvea-padic-2ed}). 
Let $\phi(t) \in \Zp[t]$ be an irreducible divisor of $\Phi_m(t)$ and $\zeta \in \ol{\Q}$ a root of $\phi(t)$. Then we have a natural isomorphism $\Zp[t]/(\phi(t))\cong \Zp[\zeta]$. 
We denote by $g(\zeta)$ the image of each $g\in \wh{\Z}[[t^{\wh{\Z}}]]$ under the map $\wh{\Z}[[t^{\wh{\Z}}]] \surj \Zp[[t^{\wh{\Z}}]] \surj \Zp[t^{\Z/n\Z}]\surj \Zp[t]/(\Phi_m(t))\surj \Zp[t]/(\phi(t))\cong \Zp[\zeta]$. 

\begin{lem} \label{not-ann} 
In the completed group ring $\wh{\Z}[[t^\wh{\Z}]]$, any element $0\neq f(t)\in \Z[t]$ is not a zero-divisor. 
\end{lem}
\begin{proof} 
It is sufficient to prove the assertion for each irreducible element $f(t)\in \Z[t]$. 
We have $\wh{\Z}[[t^\wh{\Z}]]\cong \prod_p \Zp[[t^\wh{\Z}]]$. We denote the image of elements of $\wh{\Z}[[t^\wh{\Z}]]$ in each $\Zp[[t^\wh{\Z}]]$ by the same letters. 
If $f(t)g=0$ holds for $g\in \wh{\Z}[[t^\wh{\Z}]]$, then we have $f(t)g=0$ in every $\Zp[[t^\wh{\Z}]]$. 
Since $f(t) \neq 0$ in every $\Zp[[t^\wh{\Z}]]$, 
it is sufficient to prove that $f(t)$ is not a zero-divisor in $\Zp[[t^\wh{\Z}]]$ for an arbitrary prime number $p$. \\[-2mm] 

\textbf{Case 1.} Suppose that $f(t)$ is not a cyclotomic polynomial. 
Since $\Zp$ is a unique factorization domain, by Gauss' lemma, so is $\Zp[t]$. 
Let $t^n-1=\prod_\mu \phi_\mu(t)$ denote the prime factorization in $\Zp[t]$ and let $\zeta_\mu$ be a root of $\phi_\mu(t)$ for each $\mu$. 
Since $\cap_\mu (\phi_\mu(t))=0$ in $\Zp[t^{\Z/n\Z}]$, we have a natural injection 
$\Zp[t^{\Z/n\Z}]\inj \prod_\mu \Zp[t]/(\phi_\mu(t))\cong \prod \Zp[\zeta_\mu]$, 
where each component $\Zp[\zeta_\mu]$ is an integral domain, 
and the image of $f(t)$ in $\Zp[\zeta_\mu]$ is given by $f(\zeta_\mu)$. 
Thus $\Zp[t^{\Z/n\Z}]$ injects into the product of integral domains and the image of $f(t)$ in each direct component is not zero. Therefore the image of $f(t)$ in each $\Zp[t^{\Z/n\Z}]$ is not a zero divisor, nor is it in $\Zp[[t^{\wh{\Z}}]]$. \\[-2mm] 

\textbf{Case 2.}  
Next, we prove the assertion for each cyclotomic polynomial $f(t)=\Phi_m(t)$ in three steps. 

\emph{Step 1.}  We prove the inclusion ${\rm Ann}(\Phi_m(t)^k)\subset (\Phi_m(t))$ of ideals in $\Zp[[t^{\wh{\Z}}]]$ for any $k\in \N$, where ${\rm Ann}(\Phi_m(t)^k)$ denotes the annihilator ideal of $\Phi_m(t)^k$. 
If follows from 
i) ${\rm Ann}(\Phi_m(t)^k) \subset \Ker (\mod \Phi_m(t))$ and ii) $\Ker (\mod \Phi_m(t))=(\Phi_m(t))$ proved in the following:

i) Suppose $\Phi_m(t)^kg=0$ for $g\in \Zp[[t^{\wh{\Z}}]]$. 
It is sufficient to prove $g(\xi)=0$ for every primitive $m$-th root of unity. 
We may assume $\xi=\zeta_m$. 
Let $(g_n(t))_n\in \Zp[t]^\N$ with $g=(g_n(t) \mod t^n-1)_n \in \varprojlim \Zp[t^{\Z/n\Z}]$ 
 and consider $r,n\in \N$ with $n=mp^r$. 
Since $\Phi_m(t)g_n(t)\equiv 0 \mod (t^n-1)$, 
$g_n(t)$ is divided by $\Psi_{n,m}(t):=(t^n-1)/\Phi_m(t) \in \Zp[t]$. 
The value at $\zeta_m$, which is the image in $\Zp[\zeta_m]$, 
satisfies $|\Psi_{n,m}(\zeta_m)|_p\leq |n|_p=|p^r|_p$. 
If we put $q_n(t):=g_n(t)/\Psi_{n,m}(t) \in \Zp[t]$, then $|q_n(\zeta_m)|_p\leq 1$ holds. 
Since $g(\zeta_m)=g_n(\zeta_m)$ and $\lim_{r\to \infty}|p^r|_p=0$, we have $g(\zeta_m)=0$. 

ii) The ring $\Zp[[t^{\wh{\Z}}]]$ is a compact Hausdorff topological ring with respect to the topology such that the family $\Ker(\Zp[[t^{\wh{\Z}}]]\surj \Z/p^s\Z[t^{\Z/n\Z}])\}_{s,n\in \N}$ is a fundamental neighborhood system of 0. 
The Kernel of $\mod \Phi_m(t): \Zp[[t^{\wh{\Z}}]] \surj \Zp[t]/(\Phi_m(t))$ is a closed set and contains $(\Phi_m(t))$ as a dense subset. 
Indeed, we have $(\Phi_m(t))=\Ker (\Zp[t^\Z]\inj \Zp[[t^\wh{\Z}]]\surj \Zp[t]/(\Phi_m(t)))$ in $\Zp[t^\Z]$ and the image of $\Zp[t^\Z]\inj \Zp[[t^\wh{\Z}]]$ is dense. 
Since the multiplication by $\Phi_m(t)$ is a continuous endomorphism on a compact Hausdorff space $\Zp[[t^\wh{\Z}]]$, 
it is a closed map and its image $(\Phi_m(t))$ is closed. 
Therefore we have the equality $\Ker (\mod \Phi_m(t))=(\Phi_m(t))$. \\[-2mm]

\emph{Step 2.} We obtain an inclusion of the form ``$M\subset IM$'': 
If $g\in {\rm Ann}(\Phi_m(t)^k)$, then we have $g=\Phi_m(t)h$ for some $h\in \Zp[[t^{\wh{\Z}}]]$ by Step 1. By $\Phi_m(t)^kg=\Phi_m(t)^{k+1}h=0$, we have $h \in {\rm Ann}(\Phi_m(t)^{k+1})$. 
Thus $ {\rm Ann}(\Phi_m(t)^k) \subset \Phi_m(t)({\rm Ann}(\Phi_m(t)^{k+1}))$ holds. 
Since $\{{\rm Ann}(\Phi_m(t)^k)\}_k$ is an increasing sequence with respect to inclusions, by taking $\cup_k$, we obtain $$\cup_k {\rm Ann}(\Phi_m(t)^k) \subset \Phi_m(t) (\cup_k {\rm Ann}(\Phi_m(t)^k)).$$

\emph{Step 3.} Let $g\in \Zp[[t^{\wh{\Z}}]]$ and suppose $\Phi_m(t)g=0$. 
For each $n\in \N_{>0}$, let $M$ and $I$ denote the image of $\cup_k {\rm Ann}(\Phi_m(t)^k)$ and $(\Phi_m(t))$ in $A:=\Zp[t^{\Z/n\Z}]$ respectively. Then we have $IM\subset M$. Let $g$ denote the image of $g$ in $M$ also. Since $A$ is a Noetherian ring, $M$ is a finitely generated $A$-module.  
By a well-known variant of the Nakayama--Azumaya--Krull Lemma (e.g.~\cite[Corollary 2.5]{AtiyahMacDonald}), there exists some $\alpha\in A$ satisfying $\alpha-1 \in I$ and $\alpha M=0$. 
Let $\beta \in A$ with $\alpha-1=\beta\Phi_m(t)$. 
Since $g\in M$, we have $\alpha g=(1+\beta \Phi_m(t))g=0$. By the assumption $\Phi_m(t)g=0$, we obtain $g=0$ in $A$. Therefore we have $g=0$ in $\Zp[[t^{\hat{\Z}}]]$. 

Thereby, we proved that $f(t)=\Phi_m(t)$ is not a zero-divisor in $\Zp[[t^{\hat{\Z}}]]$. 
This completes the proof of the lemma. 
\end{proof} 

\begin{lem} \label{unit} For each cyclotomic polynomial $\Phi_m(t)$ and a unit $v\in \wh{\Z}$, 
the fraction $\Phi_m(t^v)/\Phi_m(t)$ is defined and is a unit of $\wh{\Z}[[t^{\wh{\Z}}]]$. 
\end{lem}

\begin{proof} The ring $\wh{\Z}[[t^{\wh{\Z}}]]$ is a compact Hausdorff topological ring with respect to the topology such that $\{\Ker (\wh{\Z}[[t^{\wh{\Z}}]]\surj \Z/{n_1}\Z[t^{\Z/n_2\Z}]\}_{n_1,n_2}$ is a fundamental neighborhood system of $0$. 
By a similar argument to the case of $\Zp[[t^{\wh{\Z}}]]$, we have the natural surjection $\mod \Phi_m(t): \wh{\Z}[[t^{\wh{\Z}}]]\surj \wh{\Z}[t]/(\Phi_m(t))$ and the equality $(\Phi_m(t))=\Ker (\mod \Phi_m(t))$ of ideals in $\wh{\Z}[[t^{\wh{\Z}}]]$. 
Since $\Phi_m(t^v) \in \Ker (\mod \Phi_m(t))$, we have $\Phi_m(t^v)\in (\Phi_m(t))$. 
Hence $\Phi_m(t^v)=\Phi_m(t)f$ holds for some $f\in \wh{\Z}[[t^{\wh{\Z}}]]$. 
If we put $s=t^v$, then by a similar argument, we have $\Phi_m(t)=\Phi_m(s^{v^{-1}})\in \Ker(\mod \Phi_m(s))= (\Phi_m(s))=(\Phi_m(t^v))$. Hence we have $\Phi_m(t)=\Phi_m(t^v)g$ for some $g \in \wh{\Z}[[t^{\wh{\Z}}]]$. 
Now we have $\Phi_m(t)=\Phi_m(t)fg$. Since $\Phi_m(t)$ is not a zero divisor by Lemma \ref{not-ann}, 
we have $1-fg=0$. 
Therefore $f=\Phi_m(t^v)/\Phi_m(t)$ is a unit of $\wh{\Z}[[t^{\wh{\Z}}]]$. 
\end{proof} 

\begin{lem} \label{Phi} 
For polynomials $f(t), g(t) \in \Z[t]$ and a unit $v$ of $\wh{\Z}$, 
suppose the equality $(f(t))=(g(t^v))$ of ideals in $\wh{\Z}[[t^\wh{\Z}]]$. 
Then the $m$-th cyclic polynomial $\Phi_m(t)$ divides $f(t)$ if and only if it does $g(t)$. 
If $\Phi_m(t)$ divides $f(t)$, then the equality $(f(t)/\Phi_m(t))=(g(t^v)/\Phi_m(t^v))$ of ideals in $\wh{\Z}[[t^\wh{\Z}]]$ holds. 
\end{lem} 

\begin{proof}
For each $m$-th root of unity $\zeta_m \in \ol{\Q}$ and a unit $v$ of $\wh{\Z}$, 
$\zeta_m^v$ is defined and is again a primitive $m$-th root of unity. 
Hence the two equalities $g(\zeta_m)=0$ and $g(\zeta_m^v)=0$ are equivalent. 

Let $p$ be an arbitrary prime number. 
Consider the natural surjection $\wh{\Z}[[t^{\wh{\Z}}]]\surj \Zp[[t^{\wh{\Z}}]]\surj \Zp[\zeta_m]$. 
The equality $(f(t))=(g(t^{\nu}))$ of ideals in $\wh{\Z}[[t^{\wh{\Z}}]]$ yields the equality $(f(\zeta_m))=(g(\zeta_m^v))$ of ideals in $\Zp[\zeta_m]$. 

If $\Phi_m(t)$ divides $f(t)$, then we have $f(\zeta_m)=0$, $g(\zeta_m^v)=0$, and $g(\zeta_m)=0$. Hence 
$\Phi_m(t)$ divides $g(t)$. 
By Lemma \ref{not-ann}, $\Phi_m(t)$ is not a zero-divisor in $\wh{\Z}[[t^\wh{\Z}]]$. 
By Lemma \ref{unit}, $\Phi_m(t^v)/\Phi_m(t)$ is a unit of $\wh{\Z}[[t^\wh{\Z}]]$. 
Therefore we obtain the equality $(f(t)/\Phi_m(t))=(g(t^{\nu})/\Phi_m(t))=(g(t^{\nu})/\Phi_m(t^v))$ of ideals in $\wh{\Z}[[t^\wh{\Z}]]$. 
\end{proof}

\begin{lem} \label{polynomials} 
For two reciprocal polynomials $f(t),g(t)\in \Z[t]$ and a unit $v$ of $\wh{\Z}$, 
suppose the equality $(f(t))=(g(t^v))$ of ideals in $\wh{\Z}[[t^\wh{\Z}]]$. 
Then $f(t)$ and $g(t)$ coincide up to multiplication by a unit of $\Z[t^\Z]$. 
\end{lem} 

\begin{proof} 
By Lemma \ref{Phi}, we can reduce all the common cyclotomic divisors of $f(t)$ and $g(t)$. 
Note that the polynomial obtained as the quotient of two reciprocal polynomials is again reciprocal. From the reduced equality of ideals, we can derive the equality of polynomials 
by a similar method to \cite[Proposition 4.10]{BoileauFriedl2015}. 
Indeed, suppose that any cyclotomic polynomial does not divide $f(t)$ and $g(t)$. 
By Weber's proposition (Proposition \ref{Weber}), $\Z[t^{\Z/n\Z}]/(f(t))$ is a finite group with order $|R(f(t), t^n-1)|$. 
Hence we have $\wh{\Z}[t^{\Z/n\Z}]/(f(t))\cong \Z[t^{\Z/n\Z}]/(f(t))$. If we write $v=(v_n \mod n)_n$ with $v_n \in \Z$, then we have $|R(f(t), t^n-1)|=|R(g(t^{v_n}), t^n-1)|=|R(g(t),t^n-1)|$. 
By Fried's proposition (Proposition \ref{Fried}), $f(t)$ and $g(t)$ coincide up to multiplication by a unit of $\Z[t^{\Z}]$. 
\end{proof} 

In order to reduce common cyclotomic divisors, it is necessary to consider the inverse limit of modules. 
We cannot detect common non-cyclotomic divisors by their roots, 
because we can substitute only roots of unity for elements of $\wh{\Z}[[t^{\wh{\Z}}]]$. 
Fried's proposition is also essential. 

\begin{eg} \label{example} 
Fried's pair $(F(t), G(t))$ is given by 
$$F(t)=\Phi_{pq}(t)\Phi_{p^2q}(t)\Phi_{pq^2}(t), G(t)=\Phi_{p^2q^2}(t)\Phi_{pq}(t)\Phi_{pq}(t)$$ 
where $p,q$ are different prime numbers. They have same $n$-th cyclic resultants for every $n$ (\cite{Fried1988}). 
In addition, if we put 
$f(t)=F(t)^2G(t)$ and $g(t)=F(t)G(t)^2$, then $f(t)$ and $g(t)$ have the same $n$-th cyclic resultants for every $n$ and the same sets of zeros. 
By our argument, we have $(F(t)) \neq (G(t^v))$ and $(f(t))\neq (g(t^v))$ as ideals of $\wh{\Z}[[t^{\wh{\Z}}]]$ for any $v\in \wh{\Z}^*$. 
Hence they can be distinguished by comparing families of quotients $(\wh{\Z}[[s^{\wh{\Z}}]]/(F(s), s^n-1))_n$ etc. 
\end{eg}

\section{Topological lemmas} \label{ssTop}

For a discrete group $\pi$, profinite completion and Abelianization commute. 
We simply denote the profinite completion of the Abelianization of $\pi$ by $\wh{\pi}^{\rm ab}$. 
If $\pi$ is a finitely generated Abelian (additive) group, then we have $\wh{\pi}\cong \pi\otimes \wh{\Z}$. 
If $\pi$ is a finite group, then we have $\wh{\pi}\cong \pi$. 
For a finitely generated module $M$ over a Noetherian ring $R$, let ${\rm Fitt}_RM \subset R$ denote the ($0$-th) Fitting ideal of $M$ over $R$.

The following lemma tells that the profinite completions of knot groups know those of fundamental groups of finite covers over the knot exteriors.

\begin{lem} \label{hatsubgp} 
Let $\pi$ be a finitely generated discrete group, $G$ a finite group, and $\wh{\pi}\surj G$ a surjection from the profinite completion. 
Then a surjection $\pi\surj G$ is induced. 
Put $B:=\ker(\wh{\pi}\surj G)$ and $\Gamma:=\ker(\pi \surj G)$. 
Then the inclusion map $\Gamma \inj B$ induces an isomorphism $\wh{\Gamma}\congto B$ from the profinite completion. 
\end{lem} 

\begin{proof} 
The set $\mca{P}$ of normal subgroups of $\pi$ of finite index is a countable directed set with order given by the reverse of inclusions. Indeed, if $P_1, P_2\in \mca{P}$, then $P_1\cap P_2 \in \mca{P}$. 

Note that $\Gamma$ is a normal subgroup of $\pi$ of finite index. 
Let $\mca{G}$ denote the set of normal subgroups of $\Gamma$ of finite index 
and put $\mca{P}':=\mca{G}\cap \mca{P}$. 
For each $P\in \mca{P}$, there exists some $P'\in \mca{P}'$ with $P'\subset P$. 
Indeed, we may put $P'=P\cap \Gamma$. 
In addition, for each $P \in \mca{G}$, there exists some $P' \in \mca{P}'$ with $P' \subset P$. 
Indeed, we may take the intersection of all the $\pi$-conjugates of $P$ as  $P'$. 
Therefore we have natural isomorphisms $\varprojlim_{P\in \mca{P}'} \Gamma/P \cong \varprojlim_{P\in \mca{G}} \Gamma/P = \wh{\Gamma}$ and $\varprojlim_{P\in \mca{P}'} \pi/P \cong \varprojlim_{P\in \mca{P}} \pi/P = \wh{\pi}$. 
Since $\Gamma/P = \ker (\pi/P \surj G)$ holds for each $P\in \mca{P}'$, we obtain a natural isomorphism $\wh{\Gamma}\congto B$. 
\end{proof}

\begin{rem} \label{homologytorsion} 
Let $\pi$ and $\pi'$ be knot groups, $\wh{\pi}'\congto \wh{\pi}$ an isomorphisms on their profinite completions, and $\pi\surj G$ a surjection to a finite group. Let $\pi\inj \wh{\pi}\surj G$ and $\pi'\inj \wh{\pi}'\congto \wh{\pi}\surj G$ denote the induced surjections, and $X_G\to S^3-J$ and $Y_G\to S^3-K$ the corresponding covers of the knot exteriors. 
Then Lemma \ref{hatsubgp} yields a natural isomorphism $\wh{\pi}_1(X_G)\congto \wh{\pi}_1(Y_G)$. 
In addition, through the Hurewicz isomorphisms, 
an isomorphism $H_1(X_G)_{\rm tor}\congto H_1(Y_G)_{\rm tor}$ on $\Z$-torsions is induced.

Especially, if we take a representation of a knot group $\pi$ over a completed ring, then homology torsions of the corresponding inverse system of finite covers is determined by $\wh{\pi}$. 
\red{Therefore, we can study profinite rigidity of invariants associated to non-abelian covers.}  
\end{rem}

Next, we induce an isomorphism of completed Alexander modules over the completed group ring $\wh{\Z}[[t^\wh{\Z}]]$ and obtain an equality of ideals: 

\begin{lem} \label{ideals} 
Let $J$ and $K$ be knots in $S^3$ with an isomorphism $\varphi: \wh{\pi}_1(S^3-J) \congto \wh{\pi}_1(S^3-K)$ between the profinite completions of the knot groups. 
Then for some unit $v$ of $\wh{\Z}$, the equality $(\Delta_J(t^v))=(\Delta_K(t))$ of ideals in $\wh{\Z}[[t^\wh{\Z}]]$ holds. 
\end{lem} 

\begin{proof} 
Let $s$ and $t$ denote the meridians of $J$ and $K$ in $\pi_1(S^3-J)^{\rm ab}$ and $\pi_1(S^3-K)^{\rm ab}$ respectively. Then we have $\pi_1(S^3-J)^{\rm ab}=s^\Z$ and  $\pi_1(S^3-K)^{\rm ab}=t^\Z$. 
Let $\varphi: \wh{\pi}_1(S^3-J)^{\rm ab}\congto \wh{\pi}_1(S^3-K)^{\rm ab}$ denote the induced isomorphism. 
Since Abelianization and profinite completion commute, 
we have $\wh{\pi}_1(S^3-J)^{\rm ab}=s^{\wh{\Z}}$ and $\wh{\pi}_1(S^3-K)^{\rm ab}=t^{\wh{\Z}}$. 
If we denote the inverse image $\varphi^{-1}(t)$ of $t$ also by $t$, then we have $s=t^v$ for some unit $v$ of $\wh{\Z}$. 

Let $X_n\to S^3-J$ and $Y_n\to S^3-K$ denote the $\Z/n\Z$-covers, and let $M_n\to S^3$ and $N_n\to S^3$ denote the branched covers obtained as their Fox completions respectively (\cite{Fox1957}). 
The isomorphisms $\wh{\pi}_1(S^3-J)^{\rm ab}\congto \wh{\Z}; s\mapsto v$, $\wh{\pi}_1(S^3-K)^{\rm ab}\congto \wh{\Z}; t\mapsto 1$, and $\varphi$ form the following commutative diagram:  
$$\xymatrix{
\wh{\pi}_1(S^3-J) \ar@{->>}[r]  \ar[d]^{\cong}_{\varphi} & \wh{\pi}_1(S^3-J)^{\rm ab} \ar[d]^{\cong}_{\varphi} \ar[r]^{\ \ \ \ \ \cong} 
& \wh{\Z} \ar@{=}[d]\\ 
\wh{\pi}_1(S^3-K) \ar@{->>}[r] & \wh{\pi}_1(S^3-K)^{\rm ab} \ar[r]^{\ \ \ \ \ \cong} 
& \wh{\Z} \\
}$$ 
Let $\wh{\pi}_1(S^3-J)\surj \Z/n\Z$ denote the composite of the first row and the natural surjection $\wh{\Z} \surj \Z/n\Z$. Then Lemma \ref{hatsubgp} yields the natural isomorphism $\wh{\pi}_1(X_n)\cong \ker(\wh{\pi}_1(S^3-J) \to \Z/n\Z)$. 
In addition, we have well-known natural isomorphisms 
$\wh{\pi}_1(X_n)^{\rm ab} \cong \wh{H}_1(X_n)$ 
and an exact sequence $0\to s^{n\wh{\Z}}\to \wh{H}_1(X_n)\to \wh{H}_1(M_n)\to 0$ via the Hurewicz isomorphism, the Mayer--Vietoris exact sequence, and the Wang exact sequence. 
These modules and hence  $\wh{H}_1(M_n)$ $\cong \wh{H}_1(X_n)/s^{n \wh{\Z}}$ admit natural $s$-actions induced by conjugate and become $\wh{\Z}[s^{\Z/n\Z}]$-modules.  
Similarly, $\wh{H}_1(N_n)$ becomes a $\wh{\Z}[t^{\Z/n\Z}]$-module. 

Now we have $\wh{\pi}_1(S^3-J)\cong \wh{\pi}_1(S^3-K)\surj \Z/n\Z$ for each $n$. 
Hence Lemma \ref{hatsubgp} yields $\wh{\pi}_1(X_n)\cong \wh{\pi}_1(Y_n)$, as explained in Remark \ref{homologytorsion}. Since abelianization and profinite completion commute, the Hurewicz isomorphisms yield $\wh{H}_1(X_n)\cong \wh{H}_1(Y_n)$. 
Since the isomorphism $s^{n\wh{\Z}}\cong t^{n\wh{\Z}}$ commutes with other isomorphisms, we obtain a natural isomorphism $\varphi: \wh{H}_1(M_n)\cong \wh{H}_1(X_n)/s^{n\wh{\Z}} \congto \wh{H}_1(Y_n)/t^{n\wh{\Z}}\cong \wh{H}_1(N_n)$ of groups. 

We consider the $s^{\Z/n\Z}$-module $\wh{H}_1(M_n)$ as a $\wh{\Z}[t^{\Z/n\Z}]$-module via the induced isomorphism $\varphi:s^{\Z/n\Z}\congto t^{\Z/n\Z}; s\mapsto t^{v \mod n}$. 
We will verify that the induced group isomorphism $\varphi: \wh{H}_1(M_n)\congto \wh{H}_1(N_n)$ is $t^{\Z/n\Z}$-equivariant. Note that the following diagram consisting of exact rows and the induced isomorphism commutes: 
$$\xymatrix{
0 \ar[r] &
\wh{\pi}_1(X_n) \ar[r]  \ar[d]^{\cong}_{\varphi} & \wh{\pi}_1(S^3-J) \ar[d]^{\cong}_{\varphi} \ar[r]  
& \Z/n\Z \ar@{=}[d] 
\ar[r] &0
\\ 
0 \ar[r] &
\wh{\pi}_1(Y_n) \ar[r] & \wh{\pi}_1(S^3-K) \ar[r] 
& \Z/n\Z 
\ar[r] &0.\\
}$$ 
Recall that the actions of $s^{\Z/n\Z}$ and $t^{\Z/n\Z}$ are defined by the conjugate actions of lifts of elements. We denote $a^b:=b^{-1}ab$ for elements $a,b$ of a group. 
Consider the induced isomorphism $\varphi: \wh{\pi}_1(X_n)\congto \wh{\pi}_1(Y_n)$. 
Let $x,y$ be lifts of elements of $\wh{H}_1(M_n)$ to $\wh{\pi}_1(X_n)\subset \wh{\pi}_1(S^3-K)$, and let $\varphi^{-1}(r),\varphi^{-1}(u)$ be lifts of elements of $\wh{\pi}_1(S^3-K)^{\rm ab}$ to $\wh{\pi}_1(S^3-K)$ with $r,u\in \wh{\pi}_1(S^3-K)$. Then we have $\varphi(x^{\varphi^{-1}(r)} y^{\varphi^{-1}(u)})=\varphi(x)^r\varphi(y)^u$. Hence any $\bar{x},\bar{y}\in \wh{H}_1(M_n)$ and $\bar{r},\bar{u}\in t^{\Z/n\Z}$ satisfy $\varphi(\bar{r}\bar{x}+\bar{u}\bar{y})=\bar{r}\varphi(\bar{x})+\bar{u}\varphi(\bar{y})$. Thus $\varphi:\wh{H}_1(M_n)\congto \wh{H}_1(N_n)$ is an isomorphism of $\wh{\Z}[t^{\Z/n\Z}]$-modules. \\ 

Let $X_\infty\to S^3-J$ denote the $\Z$-cover. 
The Alexander module $H_1(X_\infty)$ of $J$ is a finitely generated $\Z[s^\Z]$-module with ${\rm Fitt}_{\Z[s^\Z]}H_1(X_\infty)=(\Delta_J(s))$ in $\Z[s^\Z]$. 
Namely, let $\Z[s^\Z]^q\overset{Q}{\To} \Z[s^\Z]^q \to H_1(X_\infty)\to 0$ be a finite presentation (an exact sequence) of the Alexander module with $q\in \N$ and $Q\in {\rm M}_q(\Z[s^\Z])$ (cf.~\cite[Corollary 8.C.4]{Rolfsen1990}). 
Then we have the equality $(\det Q)=(\Delta_J(t))$ of ideals in $\Z[s^{\Z}]$ by the definition of $\Delta_J(t)$. 
For each $n\in \N_{>0}$, the Wang exact sequence yields a well-known isomorphism $H_1(M_n)\cong H_1(X_\infty)/(t^n-1)H_1(X_\infty)$ of $\Z[s^{\Z/n\Z}]$-modules. Therefore we obtain a presentation $\Z[s^{\Z/n}]^q\overset{Q_n}{\To} \Z[s^{\Z/n\Z}]^q \to H_1(M_n)\to 0$ with $Q_n:=Q\mod (s^n-1) \in {\rm M}_q(\Z[t^{\Z/n\Z}])$. 
Hence ${\rm Fitt}_{\Z[s^{\Z/n\Z}]}(H_1(M_n))=
(\det Q_n)=(\Delta_J(s) \mod (s^n-1))$ holds. 
Similarly, ${\rm Fitt}_{\Z[t^{\Z/n\Z}]}(H_1(N_n))=
(\Delta_K(t) \mod (t^n-1))$  holds.

Consider the identification $\wh{\Z}[s^{\Z/n\Z}]\cong \wh{\Z}[t^{\Z/n\Z}]$ given by $s\mapsto t^v=t^{v \mod n}$. 
Then the isomorphism $\wh{H}_1(M_n)\congto \wh{H}_1(N_n)$ of $\wh{\Z}[t^{\Z/n\Z}]$-modules yields 
the equalities $$(\Delta_J(s) \mod (s^n-1))=(\Delta_J(t^v) \mod (t^n-1))=(\Delta_K(t) \mod (t^n-1))$$ of Fitting ideals in $\wh{\Z}[t^{\Z/n\Z}]$. 

In general, for each $f \in \wh{\Z}[[t^{\wh{\Z}}]]$, there is a natural isomorphism $(f)\cong \varprojlim_n (f\mod (t^n-1))$ in $\wh{\Z}[[t^{\wh{\Z}}]]$. 
Indeed, let $K_n$ denote the kernel of the restriction $\mod (t^n-1): (f)\surj (f \mod (t^n-1))$ of the natural surjection for each $n$. 
We may assume that $n$ runs through the ordered subset $\N':=\{m!\mid m\in \N\}$ of $\N$. 
Since $\{K_n\}_n$ is a surjective system, it satisfies the Mittag-Leffler condition and $\varprojlim_n^1 K_n=0$ holds. Together with $\varprojlim_n K_n=0$, the isomorphism is induced (e.g.~\cite[Section 1]{Jannsen1988}). 

Therefore, taking the inverse limit in the equality of the Fitting ideals in $\wh{\Z}[t^{\Z/n\Z}]$, we obtain the equality $(\Delta_J(t^v))=(\Delta_K(t))$ of ideals in $\wh{\Z}[[t^{\wh{\Z}}]]$. 
\end{proof}

In the rest of this section, we verify that the equality of ideals obtained in Lemma \ref{ideals} can be interpreted as that of Fitting ideals of completed Alexander modules. 
Let the notation be as in Lemma \ref{ideals}. We define \emph{the completed Alexander module} of $J\subset S^3$ by $\mca{H}_J:=\varprojlim_n \wh{H}_1(M_n)$. 
\begin{lem} The completed Alexander module $\mca{H}_J$ is a finitely generated $\wh{\Z}[[s^{\wh{\Z}}]]$-module with the Fitting ideal $(\Delta_J(s))$ in $\wh{\Z}[[s^{\wh{\Z}}]]$. 
\end{lem}
\begin{proof} 

Let the notation be as in the proof of Lemma \ref{ideals} and consider the finite presentation of $H_1(M_n)$ by $Q_n \in {\rm M}_q(\Z[t^{\Z/n\Z}])$. 
Since $\wh{\Z}$ is flat over $\Z$, the functor $\otimes \wh{\Z}$ is exact for modules. Hence we have a presentation  
$\wh{\Z}[s^{\Z/n\Z}]^q\overset{\wh{Q}_n}{\To} \wh{\Z}[s^{\Z/n\Z}]^q \to \wh{H}_1(M_n)\to 0$ of $\wh{H}_1(M_n)$, where $\wh{Q}_n=Q_n$ as matrices. 

We may assume that $n$ runs through $\N'=\{m!\mid m\in \N\}$. 
Taking the inverse limit, we obtain an exact sequence $\wh{\Z}[[s^{\wh{\Z}}]]^q\overset{\wh{Q}}{\To} \wh{\Z}[[s^{\wh{\Z}}]]^q \to \mca{H}_J\to \varprojlim_n^1 \Ker \wh{Q}_n$ with $\wh{Q}=(Q_n)_n \in {\rm M}_q(\wh{\Z}[[s^{\wh{\Z}}]])$. Since $(\Ker \wh{Q}_n)_n$ is a surjective system, we have $\varprojlim_n^1 \Ker \wh{Q}_n=0$. Thus we obtain a finite presentation 
$$\wh{\Z}[[s^{\wh{\Z}}]]^q\overset{\wh{Q}}{\To} \wh{\Z}[[s^{\wh{\Z}}]]^q \to \mca{H}_J\to0$$ 
of $\mca{H}_J$. 
Therefore $\mca{H}_J$ is a finitely generated $\wh{\Z}[[s^{\wh{\Z}}]]$-module. (This fact can be obtained also in an abstract way by using the topological Nakayama's lemma.)

Since $(\det Q_n)=(\Delta_J(t) \mod (t^n-1))$ in each $\wh{\Z}[s^{\Z/n\Z}]$, 
we obtain the equalities 
${\rm Fitt}_{\wh{\Z}[[s^{\wh{\Z}}]]} \mca{H}_J=(\det \wh{Q})=\varprojlim_n (\det Q_n) =(\Delta_J(s))$ of ideals in $\wh{\Z}[[s^{\wh{\Z}}]]$. 
\end{proof}

The module $\mca{H}_J$ is a $\wh{\Z}[[t^{\wh{\Z}}]]$-module under the identification $\wh{\Z}[[s^{\wh{\Z}}]]\cong \wh{\Z}[[t^{\wh{\Z}}]]; s\mapsto t^v$. 
We put $\mca{H}_K:=\varprojlim_n \wh{H}_1(N_n)$. 
Since the isomorphisms $\wh{H}_1(M_n)\congto \wh{H}_1(N_n)$ of $\wh{\Z}[[t^{\Z/n\Z}]]$-modules are compatible with the inverse systems, 
we obtain an isomorphism $\mca{H}_J\congto \mca{H}_K$ of $\wh{\Z}[[t^{\wh{\Z}}]]$-modules. 
This yields the equality $(\Delta_J(t^v))=(\Delta_K(t))$ of Fitting ideals in $\wh{\Z}[[t^{\wh{\Z}}]]$, 
which coincides with the one we obtained in Lemma \ref{ideals}. 

\section{Proof of Theorem \ref{theorem}} 
\label{ssPrf} 

\begin{proof}[
Proof of Theorem \ref{theorem}] Since the Alexander polynomials of knots are reciprocal up to multiplication by units of $\Z[t^{\Z}]$, the theorem follows immediately from Lemmas \ref{polynomials} and \ref{ideals}. 
\end{proof}

An isomorphism $\varphi: \wh{\pi}_1(S^3-J)\congto \wh{\pi}_1(S^3-K)$ does not necessarily yield an isomorphism 
$\psi: \wh{\pi}_1(S^3-J)^{\rm ab}\congto \wh{\pi}_1(S^3-K)^{\rm ab}$ 
sending the meridian of $J$ to that of $K$. 
For a reciprocal polynomial $g(t)$ in $\Z[t]$ and a unit $v$ of $\wh{\Z}$, the equality $(g(t))=(g(t^v))$ does not necessarily hold. 
Therefore Lemma \ref{ideals} seems to be the best we can say along this direction in Section \ref{ssTop}. 
In addition, even if we have 
an isomorphism $\psi$ and the equality, we still need the algebraic argument in Section \ref{ssAlg} to determine polynomials. 

\red{We finally remark that indeed we only needed an isomorphism of the pro-metabelian completions of knot groups to prove the coincidence of Alexander polynomials.}

We will try to apply our method to twisted Alexander polynomials of knots or invariants of links in our future works. 

\section*{Acknowledgments} 
I would like to express my gratitude to Oliver Br\"aunling, 
\red{Michel Boileau, Stefan Friedl,} 
Jonathan Hillman, Teruhisa Kadokami, Kensaku Kinjo, Tomoki Mihara, Yasushi Mizusawa, Takayuki Morisawa, Hiroaki Nakamura, 
\red{Mark Pengitore, Alan Reid,} 
Yuji Terashima, Lorenzo Traldi, and 
\red{Gareth Wilkes} 
for fruitful conversations. 
I am grateful to the anonymous referees for a lot of helpful feedback. 
I am also grateful to my family and friends for essential support. 
I was partially supported by Grant-in-Aid for JSPS Fellows (
25-2241). 


\bibliographystyle{amsalpha}
\providecommand{\bysame}{\leavevmode\hbox to3em{\hrulefill}\thinspace}
\providecommand{\MR}{\relax\ifhmode\unskip\space\fi MR }
\providecommand{\MRhref}[2]{%
  \href{http://www.ams.org/mathscinet-getitem?mr=#1}{#2}
}
\providecommand{\href}[2]{#2}


\ \\
\noindent 
Jun Ueki \\
Graduate School of Mathematical Sciences, The University of Tokyo\\
3-8-1, Komaba, Meguro-ku, Tokyo, 153-8914, Japan \\ 
E-mail: \url{uekijun46@gmail.com} 

\end{document}